\documentclass[11pt, reqno]{amsart}

\usepackage[usenames,dvipsnames]{color}

\usepackage[OT2, T1]{fontenc}
\usepackage{url}
\usepackage{amsmath}
\usepackage{array}
\usepackage{graphicx}
\usepackage{amssymb}
\usepackage{amsthm}
\definecolor{link}{RGB}{11,0,128}
\usepackage[colorlinks=true,citecolor=NavyBlue,linkcolor=OliveGreen,urlcolor=link]{hyperref}
\usepackage{hyperref}
\usepackage{paralist}
\usepackage{colonequals}			
\usepackage{color}
\usepackage{mathabx}			
\usepackage{amscd}
\usepackage[all,cmtip]{xy}			
\usepackage{sseq}				
\usepackage{verbatim}
\usepackage{parskip}
\usepackage{microtype}
\usepackage[in]{fullpage}
\usepackage{mathrsfs} 			
\usepackage{cleveref}
\usepackage{enumitem}
\usepackage{moreenum}			
\usepackage[alphabetic,lite]{amsrefs} 	
\usepackage{stmaryrd}			
\usepackage[foot]{amsaddr}
\usepackage{subfiles}
\usepackage{ragged2e}
\usepackage{stackengine}

\newcommand{\gA}{\alpha}

\newcommand{\bP}{\mathbb{P}}

\newcommand{\cE}{\mathcal{E}}

\newcommand{\cG}{\mathcal{G}}

\newcommand{\cO}{\mathcal{O}}

\newcommand{\cX}{\mathcal{X}}

\newcommand{\sO}{\mathscr{O}}

\DeclareMathOperator{\Gr}{Gr}			
\DeclareMathOperator{\Ker}{Ker}		
\DeclareMathOperator{\Lie}{Lie}		
\DeclareMathOperator{\Res}{Res}		
\DeclareSymbolFont{cyrletters}{OT2}{wncyr}{m}{n}
\DeclareMathSymbol{\Sha}{\mathalpha}{cyrletters}{"58}	
\DeclareMathOperator{\Spa}{Spa}		
\DeclareMathOperator{\Spec}{Spec}		

\newcommand{\ce}{\colonequals}

\newcommand{\dR}{{\mathrm{dR}}}		

\newcommand{\et}{\mathrm{\acute{e}t}}	

\newcommand{\isomto}{\overset{\sim}{\longrightarrow}}

\newcommand{\llb}{\llbracket}			
\newcommand{\llp}{(\!(}			
	
\newcommand{\ra}{\rightarrow}

\newcommand{\rrb}{\rrbracket}			
\newcommand{\rrp}{)\!)}			

\newcommand{\surjects}{\twoheadrightarrow}
\newcommand{\tensor}{\otimes} 			

\newcommand{\wt}{\widetilde}

\providecommand{\up}[1]{{\upshape(}#1{\upshape)}}
\providecommand{\uref}[1]{{\upshape\ref{#1}}}
\providecommand{\uS}{{\upshape\S}}
\providecommand{\f}[2]{\frac{#1}{#2}}

\renewcommand{\b}{\textbf}
\providecommand{\ucolon}{{\upshape:} }

\newcommand{\brems}{\begin{rems} \hfill \begin{enumerate}[label=\b{\thenumberingbase.},ref=\thenumberingbase]}

\newcommand{\erems}{\end{enumerate} \end{rems}}
\newcommand{\begs}{\begin{egs} \hfill \begin{enumerate}[label=\b{\thenumberingbase.},ref=\thenumberingbase]}

\newcommand{\eegs}{\end{enumerate} \end{egs}}
\newcommand{\m}{\item}
\newcommand{\bsm}{\begin{smallmatrix}}
\newcommand{\esm}{\end{smallmatrix}}
\newcommand{\blem}{\begin{lemma}}
\newcommand{\elem}{\end{lemma}}

\newcommand{\bconj}{\begin{conj}}
\newcommand{\econj}{\end{conj}}
\newcommand{\bprob}{\begin{Problem}}
\newcommand{\eprob}{\end{Problem}}

\newcommand{\bq}{\begin{Q}}
\newcommand{\eq}{\end{Q}}
\newcommand{\benum}{\begin{enumerate}[label={{\upshape(\alph*)}}]}
\newcommand{\benuma}{\begin{enumerate}[label={{\upshape(\arabic*)}}]}
\newcommand{\benumb}{\begin{enumerate}[label={{\upshape\b{\arabic*.}}}]}
\newcommand{\benumr}{\begin{enumerate}[label={{\upshape(\roman*)}}]}
\newcommand{\eenum}{\end{enumerate}}
\newcommand{\bitem}{\begin{itemize}}
\newcommand{\eitem}{\end{itemize}}
\newcommand{\bc}{}
\newcommand{\bd}{\begin{defn}}
\newcommand{\ed}{\end{defn}}

\newcommand{\beg}{\begin{eg}}
\newcommand{\eeg}{\end{eg}}

\newcommand{\bcl}{\begin{claim}}
\newcommand{\ecl}{\end{claim}}

\newcommand{\q}{\quad}
\providecommand{\qxq}[1]{\quad\text{#1}\quad}
\providecommand{\qx}[1]{\quad\text{#1}}

\newcommand{\tst}{\textstyle}
\newcommand{\ba}{\begin{aligned}}
\newcommand{\ea}{\end{aligned}}
\newcommand{\be}{\begin{equation}}
\newcommand{\ee}{\end{equation}}
\newcommand{\bpf}{\begin{proof}}
\newcommand{\epf}{\end{proof}}
\newcommand{\bthm}{\begin{thm}}
\newcommand{\ethm}{\end{thm}}
\newcommand{\bprop}{\begin{prop}}
\newcommand{\eprop}{\end{prop}}
\newcommand{\bcor}{\begin{cor}}
\newcommand{\ecor}{\end{cor}}
\newcommand{\brem}{\begin{rem}}
\newcommand{\erem}{\end{rem}}

\usepackage{aliascnt}
\newaliascnt{numberingbase}{subsection}
\numberwithin{equation}{numberingbase}

\newtheoremstyle{thms}{0.5em}{0.5em}{\itshape}{}{\bfseries}{.}{ }{}
\theoremstyle{thms}
\newtheorem{conj}[numberingbase]{Conjecture}

\newtheorem{cor}[numberingbase]{Corollary}

\newtheorem{lemma}[numberingbase]{Lemma}

\newtheorem{prop}[numberingbase]{Proposition}

\newtheorem{Q}[numberingbase]{Question}
\newtheorem{thm}[numberingbase]{Theorem}

\newtheoremstyle{claims}{0.5em}{0.5em}{}{}{\itshape}{.}{ }{}
\theoremstyle{claims}
\newtheorem{claim}[equation]{Claim}

\newtheoremstyle{defs}{0.5em}{0.5em}{}{}{\bfseries}{.}{ }{}
\theoremstyle{defs}
\newtheorem{defn}[numberingbase]{Definition}

\newtheorem{eg}[numberingbase]{Example}
\newtheorem*{egs}{Examples}
\newtheorem{rem}[numberingbase]{Remark}

\newtheorem*{rems}{Remarks}

\Crefname{claim}{Claim}{Claims}
\Crefname{bclaim}{Claim}{Claims}
\Crefname{sublemma}{Lemma}{Lemmas}
\Crefname{conj}{Conjecture}{Conjectures}
\Crefname{cor}{Corollary}{Corollaries}
\Crefname{defn}{Definition}{Definitions}
\Crefname{eg}{Example}{Examples}
\Crefname{prop}{Proposition}{Propositions} 
\Crefname{Q}{Question}{Questions}
\Crefname{rem}{Remark}{Remarks}
\Crefname{thm}{Theorem}{Theorems}
\Crefname{Theorem}{Theorem}{Theorems}
\Crefname{variant}{Variant}{Variants}
\Crefname{caution}{Caution}{Cautions}

\theoremstyle{thms}
\newtheorem{thm-tweak}[subsection]{Theorem}
\Crefname{thm-tweak}{Theorem}{Theorems}
\newtheorem{lemma-tweak}[subsection]{Lemma}
\Crefname{lemma-tweak}{Lemma}{Lemmas}
\newtheorem{cor-tweak}[subsection]{Corollary}
\Crefname{cor-tweak}{Corollary}{Corollaries}
\newtheorem{prop-tweak}[subsection]{Proposition}
\Crefname{prop-tweak}{Proposition}{Propositions} 
\newtheorem{conj-tweak}[subsection]{Conjecture}
\Crefname{conj-tweak}{Conjecture}{Conjectures} 
\newtheorem{q-tweak}[subsection]{Question}
\Crefname{q-tweak}{Question}{Questions} 

\theoremstyle{defs}
\newtheorem{defn-tweak}[subsection]{Definition}
\Crefname{defn-tweak}{Definition}{Definitions}
\newtheorem{eg-tweak}[subsection]{Example}
\Crefname{eg-tweak}{Example}{Examples}
\newtheorem*{rems-tweak}{Remarks}
\newtheorem{rem-tweak}[subsection]{Remark}
\Crefname{rem-tweak}{Remark}{Remarks}

\newtheoremstyle{subsection-tweak}
   {2pt}
   {3pt}%
   {}
   {}%
   {\bfseries}
   {}%
   {.5em}
   {\thmnumber{\@{#1}{}\@{#2}.}%
    \thmnote{~{\bfseries#3.}}}    
    
\theoremstyle{subsection-tweak}
\newtheorem{pp}[numberingbase]{}
\newcommand{\bpp}{\begin{pp}}
\newcommand{\epp}{\end{pp}}

\theoremstyle{subsection-tweak}
\newtheorem{pp-tweak}[subsection]{}

\makeatletter
\def\@tocline#1#2#3#4#5#6#7{
    \begingroup 
    \@ifempty{#4}{}{}

    \parindent\z@ \leftskip#3\relax \advance\leftskip\@tempdima\relax
    #5\hskip-\@tempdima
      \ifcase #1
       \or\or \hskip 2em \or \hskip 1em \else \hskip 3em \fi%
      #6\nobreak\relax
    \dotfill\hbox to\@pnumwidth{\@tocpagenum{#7}}\par
    \nobreak
    \endgroup
 }
 \def\l@section{\@tocline{1}{0pt}{1pc}{}{}}

\renewcommand{\tocsection}[3]{%
  \indentlabel{\@ifnotempty{#2}{\makebox[1.3em][l]{%
    \ignorespaces#1 \bfseries{#2}.\hfill}}}\bfseries{#3}
    \vspace{-5pt}}

\renewcommand{\tocsubsection}[3]{%
  \indentlabel{\@ifnotempty{#2}{\hspace*{-0.5em}\makebox[2.1em][l]{%
    \ignorespaces#1#2.\hfill}}}#3
    \vspace{-5pt}}
   
\makeatother 

\makeatletter 
\newcommand\appendix@section[1]{%
  \refstepcounter{section}%
  \orig@section*{Appendix \@Alph\c@section. #1}%
}
\let\orig@section\section
\g@addto@macro\appendix{\let\section\appendix@section}
\makeatother

\makeatletter
\@namedef{subjclassname@2020}{%
  \textup{2020} Mathematics Subject Classification}
\makeatother

\author{K\k{e}stutis \v{C}esnavi\v{c}ius $^{(1)}$}
\author{Alex Youcis $^{(2)}$}
\address[1]{\scriptsize CNRS, Universit\'{e} Paris-Saclay,   Laboratoire de math\'{e}matiques d'Orsay, F-91405, Orsay, France}
\address[2]{\scriptsize National University of Singapore, Level 4, Block S17, 10 Lower Kent Ridge Road, Singapore 119076}
\email[1]{\scriptsize kestutis@math.u-psud.fr}
\email[2]{\scriptsize alex.youcis@gmail.com}

\date{\today}

\usepackage{nicefrac}

\begin{document}


\title{The analytic topology suffices for the $B_\dR^+$-Grassmannian}

\maketitle

\begin{abstract} 
The $B_\dR^+$-affine Grassmannian was introduced by Scholze in 
the context of the geometric local Langlands program in mixed characteristic and is the Fargues--Fontaine curve analogue of the equal characteristic Beilinson--Drinfeld affine Grassmannian. For a reductive group $G$, it is defined as the \'{e}tale (equivalently, $v$-) sheafification of the presheaf quotient $LG/L^+G$ of the $B_\dR$-loop group $LG$ by the $B_\dR^+$-loop subgroup $L^+G$. We combine algebraization and approximation techniques with known cases of the Grothendieck--Serre conjecture to show that the analytic topology suffices for this sheafification, more precisely, that the $B_\dR^+$-affine Grassmannian agrees with the analytic sheafification of the aforementioned presheaf quotient $LG/L^+G$.
\end{abstract}

\hypersetup{
    linktoc=page,     
}
\renewcommand*\contentsname{}
\q\\
\tableofcontents

\newcommand{\Bdrp}{B_{\mathrm{dR}}^+}

\newcommand{\Bdr}{B_{\mathrm{dR}}}

 \newcommand{\alex}[1]{\footnote{\textcolor{purple}{Alex:} #1}}
 \newcommand{\kestutis}[1]{\footnote{\textcolor{NavyBlue}{Kestutis:} #1}}

\newcommand{\Z}{\mathbb{Z}}
\newcommand{\R}{\mathbb{R}}

\newcommand{\mc}[1]{\mathcal{#1}}
\newcommand{\cat}[1]{\mathbf{#1}}
\newcommand{\triv}{\mathrm{triv}}

\section{The $\Bdrp$-affine Grassmannian} \label{sec:recall}

For a reductive group $G$ over a ring $R$, the Beilinson--Drinfeld affine Grassmannian $\Gr_G$ plays an important role in the geometric Langlands program, as well as in other fields that feature reductive groups and their torsors. Letting $LG$ (resp.,~$L^+G$) be the \emph{loop functor} (resp.,~its \emph{positive loop subfunctor}) that sends a variable $R$-algebra $A$ to $G(A\llp t \rrp)$ (resp.,~to $G(A\llb t\rrb)$), the affine Grassmannian $\Gr_G$ is defined as the \'{e}tale (equivalently, fpqc) sheafification of the presheaf quotient $LG/L^+G$. By the recent results \cite{Gr-presheaf}*{Theorems 2.5 and~3.4}, based on the study of $G$-torsors over $\bP^1_A$, Zariski sheafification gives the same result and, if $G$ is totally isotropic (for instance, quasi-split), then no sheafification is needed at all: then $\Gr_G$ already agrees with the presheaf quotient $LG/L^+G$. 

In his Berkeley lectures \cite{SW20}, Scholze adapted the definition of Beilinson--Drinfeld to the then-emergent geometric \emph{local} Langlands program, and subsequently with Fargues applied it in their elaboration of this program in \cite{FS24}. To review his definition, we let $K$ be a nonarchimedean local field, let $G$ be a smooth affine group scheme defined either over $K$ or over its ring of integers $\cO_K$, 
and recall that the $B_\dR^+$-affine Grassmannian $\Gr_G^{B_\dR^+}$ is a functor on the category of perfectoid  $\cO_K$-algebra pairs $(A, A^+)$. 
For such a pair, we let $(A^\flat, A^{\flat+})$ denote its tilt, let $\varpi^\flat \in A^{\flat+}$ be a pseudouniformizer such that $\omega\ce(\omega^\flat)^\sharp$ satisfies $\omega^p\mid p$ in $A^+$ (see \cite[Lemma 6.2.2]{SW20} or \cite{flat-purity}*{Section~2.1.2}), let $W_{\cO_K}(A^{\flat+})$ denote the $\cO_K$-ramified Witt vectors of $A^{\flat+}$, and let 
\[
\tst I\ce \Ker\left(W_{\cO_K}(A^{\flat+})\surjects A^+\right)
\]
be the kernel of the Fontaine $\cO_K$-algebra map sending any Teichm\"uller $[a]$ to $a^\sharp$ (compare with~\cite{flat-purity}*{Equation~(2.1.1.1)}). In the $B_\dR^+$ context, the role of the formal power series ring $A\llb t \rrb$ is played by 
 \begin{equation*}
\tst    \Bdrp(A) \ce \varprojlim_n (W_{\cO_K}(A^{\flat+})[\f{1}{[\omega^\flat]}]/I^n)
 \end{equation*}
 compare with \cite[page~138]{SW20}. This notation is slightly abusive because $B_\dR^+(A)$ does depend on $K$, although not on $A^+$ nor on $\varpi^\flat$. 
  The ideal $I$ is functorial in $(A, A^+)$, principal, and generated by a nonzerodivisor $\xi$, and the role of the Laurent power series ring $A\llp t \rrp$ is played~by
\[
\tst\Bdr(A)\ce\Bdrp(A)[\f{1}{I}].
\]
In the following special cases, we can be slightly more explicit.
\benuma
\m \label{m:case-1}
The field $K$ is of characteristic $0$ and $A$ is a $K$-algebra: then we may choose $\varpi^\flat$ such that $\varpi^p$ is a unit multiple of $p$ (see \cite{flat-purity}*{Section 2.1.2}), so $B_\dR^+(A)$ and $B_\dR(A)$ are $K$-algebras. 

\m
The field $K$ is of characteristic $0$ and $A$ is an algebra over its residue field $k$: then $A^{\flat+} \cong A^+$, so $I = (\pi)$ where $\pi \in \cO_K$ is a uniformizer, and $B_\dR^+(A) = W_{\cO_K}(A)$ with $B_\dR(A) = W_{\cO_K}(A)[\f1{\pi}]$.

\m \label{m:case-3}
The field $K$ is of characteristic $p > 0$. Then $\cO_K \simeq k\llb \zeta\rrb$, the functor $W_{\cO_K}(-)$ is the completed the tensor product over $k$ with $k\llb \zeta \rrb$, and $B_\dR^+(A) \simeq A\llb \pi - \zeta\rrb$ with $B_\dR(A) \simeq A\llp \pi - \zeta\rrp$, where $\pi \in \cO_K$ is a uniformizer; in terms of this presentation, the ideal $I$ is generated by $\pi - \zeta$. For the sake of uniformity of discussion, we do not exclude this case, but it will offer nothing new relative to the setting of the Beilinson--Drinfeld affine Grassmannian reviewed above. 
\eenum
We stress that the cases \ref{m:case-1}--\ref{m:case-3} are nonexhaustive, see, for instance, \cite{SW20}*{Example 6.1.5 4.}. 

In the $B_\dR^+$ context, the \emph{loop functor} (resp.,~its \emph{positive loop subfunctor}) is defined by
\[
LG \colon (A, A^+) \mapsto G(B_\dR(A)) \qxq{\upshape{(resp.,~by}} L^+G \colon (A, A^+) \mapsto G(B^+_\dR(A))),
\]
granted either that $G$ is defined over $\cO_K$ or that one restricts to $(A, A^+)$ with $A$ a $K$-algebra. We have reused the notation $LG$ and $L^+G$ because the power series context will henceforth play no role.

The \emph{$B_\dR^+$-affine Grassmannian} $\Gr_G^{B_\dR^+}$ is defined as the \'{e}tale sheafification of the presheaf quotient $LG/L^+G$, compare with \cite{SW20}*{Definition 19.1.1}. Our goal in this article is to show that for reductive $G$, the sheafification for the much coarser analytic topology gives the same result, namely, that $\Gr_G^{B_\dR^+}$ agrees with the analytic sheafification of the presheaf quotient $LG/L^+G$, see \Cref{thm:main}~below. 

In the view of the analogy with the Beilinson--Drinfeld affine Grassmannian, it could be that, at least for totally isotropic $G$, no sheafification is needed at all, namely, that $\Gr_G^{B_\dR^+}$ is the  presheaf quotient $LG/L^+G$. However, we do not know how to approach this, nor how to find a counterexample.

\section{The modular description} \label{sec:modular}


Our proof that the analytic sheafification suffices will hinge on the following modular description of the $B_\dR^+$-affine Grassmannian. This modular description has already been given in \cite[Proposition~19.1.2]{SW20}, although we prefer the slightly different argument for it given below.


\begin{prop}\label{prop:Gr-is-v-sheaf}
For a nonarchimedean local field $K$ and a smooth affine group scheme $G$ defined either over $K$ or over $\cO_K$, letting $(A, A^+)$ range over perfectoid  $\cO_K$-algebra pairs \up{resp.,~such that $A$ is a $K$-algebra if $G$ is only defined over $K$}, we have the following functorial modular~interpretation\ucolon
\[
    \Gr^{B_\dR^+}_G(A)\cong \left\{(\cE,\iota)\colon \begin{aligned}&\  \cE\text{ is a }G\text{-torsor over }\Bdrp(A),\\ & \  \iota \in  \cE(B_\dR(A)) \text{ is a trivialization over $B_\dR(A)$}
    \end{aligned}\right\}/\sim.
\]
 Moreover, $\Gr^{B_\dR^+}_G$ is a sheaf for the $v$-topology and it contains the presheaf quotient $LG/L^+G$ as the subfunctor that parametrizes those pairs $(\cE, \iota)$ for which $\cE$ is a trivial torsor.
\end{prop}

\begin{proof}  
Let us temporarily write $\Gr_G'$ for the functor defined by the displayed modular description. Since $LG$ parametrizes the possible $\iota$ for the trivial $G$-torsor and $L^+G$ parametrizes the automorphisms of the trivial $G$-torsor over $B_\dR^+(A)$, the presheaf quotient $LG/L^+G$ is identified with the subfunctor of $\Gr_G'$ that parametrizes those $(\cE, \iota)$ for which $\cE$ is a trivial $G$-torsor. Thus, all we need to show is that $\Gr'_G$ is a $v$-sheaf (in particular, an \'{e}tale sheaf) and that each $(\cE, \iota)$ lies in $LG/L^+G$ \'etale locally on $\Spa(A, A^+)$, in other words, that $\cE$ trivializes \'{e}tale locally on $\Spa(A, A^+)$.

For the latter, of course, $\cE$ inherits smoothness over $B_\dR^+(A)$ (and also affineness) from $G$, so to trivialize $\cE$ over $B_\dR^+(A')$ for a perfectoid $\cO_K$-algebra pair $(A', A'^+)$ over $(A, A^+)$, it suffices to arrange that $\cE(A') \neq \emptyset$, that is, to trivialize $\cE|_{A'}$. In other words, we need to trivialize $\cE|_A$ \'etale locally on $\Spa(A, A^+)$, knowing that, by the $A$-smoothness of $\cE|_A$, the latter trivializes \'{e}tale locally on $\Spec(A)$. By \cite{SAG}*{Proposition 1.2.3.4}, the \'etale topos $(\mathrm{Shv}(\Spec(A)_\et), \sO_{\Spec(A)_\et})$ ringed by the structure sheaf is final among all the strictly Henselian locally ringed topoi $(\cX, \sO_\cX)$ equipped with a ring homomorphism $A \ra \Gamma(\cX, \sO_\cX)$. Since the strictly local rings of an adic space are strictly Henselian, by \cite{SAG}*{Remark 1.2.2.10} (with \cite{Hub96}*{Propositions~2.5.5 and~2.5.17}), the \'{e}tale topos $(\mathrm{Shv}(\Spa(A, A^+)_\et), \sO_{\Spa(A,\, A^+)_\et})$ of the adic space $\Spa(A, A^+)$ ringed by the structure sheaf is an example of an $(\cX, \sO_\cX)$ as above. It follows that there is a morphism of ringed topoi 
\[
(\mathrm{Shv}(\Spa(A, A^+)_\et), \sO_{\Spa(A,\, A^+)_\et}) \rightarrow (\mathrm{Shv}(\Spec(A)_\et), \sO_{\Spec(A)_\et}).
\]
Thus, an \'etale cover of $\Spec(A)$ that trivializes $\cE|_A$ pulls back to an \'etale cover of $\Spa(A, A^+)_\et$, to the effect that $\cE|_A$ also trivializes locally on $\Spa(A, A^+)_\et$,  as desired.\footnote{A more concrete way to trivialize $\cE|_{A}$ \'etale locally on $\Spa(A, A^+)_\et$ is as follows. For a fixed $x \in \Spa(A, A^+)$,  let $\Spa(A_i, A_i^+) \subset \Spa(A, A^+)$ range over the rational opens containing $x$. By, for example, \cite[Example~5.2]{Sch22}, 
the residue pair $(k_x, k_x^{+})$ is perfectoid and, letting $\varpi \in A^+$ be a pseudouniformizer, $k_x^{+} \cong \left(\varinjlim_i A_i^{+}\right)^{\wedge}$ with $k_x \cong k_x^+[\f1\varpi]$, where the completion is $\varpi$-adic. Since $k_x$ is a field and $\cE$ is smooth and nonempty, $\cE(\wt k) \neq \emptyset$ for some finite separable field extension $\wt{k}/k_x$. Each $A_i^+$ is $\varpi$-Henselian and $A_i \cong A_i^+[\f1\varpi]$ (resp.,~$k_x \cong k_x^+[\f1\varpi]$), so $\varinjlim_i A_i^{+}$ is $\varpi$-Henselian 
with $\varinjlim_i A_i \cong (\varinjlim_i A_i^{+})[\f1\varpi]$. Thus, by \cite[Corollary~2.1.20]{Hitchin-torsors} (with $B \ce \varinjlim_i A_i$ and $B' \ce \varpi (\varinjlim_i A_i^{+})$ there) and a limit argument as in \cite{KL15}*{Remark 1.2.9}, there are an $i$ and a finite \'{e}tale $A_i$-algebra $\wt{A}$ with $\wt{A} \tensor_{A_i} k_x \simeq \wt{k}$. The restriction of scalars $\wt{\cE} \ce \Res_{\wt{A}/A_i}(\cE_{\wt{A}})$ is a smooth affine $A_i$-scheme, see \cite{BLR90}*{Section 7.6, Propositions 2 and 5, Theorem~4}, and, by construction, $\wt{\cE}(k_x) \cong \cE(\wt{k}) \neq \emptyset$. Therefore, by \cite{Hitchin-torsors}*{Theorem 2.2.2}, at the cost of enlarging $i$, also $\wt{\cE}(A_i) \cong \cE(\wt{A}) \neq \emptyset$. This implies the desired triviality of $\cE$ over some \'etale neighborhood of $x$.}

For the remaining $v$-sheaf property of $\Gr_G'$, we will use the trivialization $\iota$. Firstly, thanks to $\iota$ and the ideal $I$ from \S\ref{sec:recall} being generated by a nonzerodivisor, the objects of the prestack in groupoids
\[
\tst (A, A^+) \mapsto \left\{(\cE,\iota)\colon \begin{aligned}&\  \cE\text{ is a }G\text{-torsor over }\Bdrp(A),\\ & \  \iota \in  \cE(B_\dR(A)) \text{ is a trivialization over $B_\dR(A)$}
    \end{aligned}\right\}
\]
have no nontrivial automorphisms. Thus, it suffices to show that this prestack is a $v$-stack, in fact, that $v$-descent is uniquely effective for its objects. To this end, we consider maps
\[
(A, A^+) \ra (A', A'^+) \rightrightarrows (A'', A''^+)
\]
of perfectoid  $\cO_K$-algebra pairs such that the first one induces a $v$-cover on adic spectra whose self-(fiber product) is $\Spa(A'', A''^+)$, and we consider a pair $(\cE', \iota')$ over $B_\dR^+(A')$ equipped with a descent datum with respect to this cover. Since the ideal $I$ is functorial in $(A, A^+)$ (see \S\ref{sec:recall}), the reduction modulo $I$ of this descent datum equips the $G$-torsor $\cE'|_{B_\dR^+(A')/I}$ over $A'$ with a descent datum with respect to the $v$-cover $(A, A^+) \ra (A', A'^+)$. By the Tannakian formalism for $G$-torsors over perfectoid spaces \cite{SW20}*{Theorem 19.5.2} combined with the $v$-descent for vector bundles on perfectoid spaces \cite{SW20}*{Corollary 17.1.9}, this last descent datum is uniquely effective, so $\cE'|_{B_\dR^+(A')/I}$ descends to a unique $G$-torsor $E$ over $A$. 

The invariance under Henselian pairs for $G$-torsors \cite{Hitchin-torsors}*{Theorem 2.1.6}\footnote{Or already its earlier version \cite{GR03}*{Theorem 5.8.14}.} ensures that $E$ lifts uniquely to a $G$-torsor $\cE$ over $B_\dR^+(A)$ and, by uniqueness, $\cE|_{B_\dR^+(A')} \simeq \cE'$ compatibly with the canonical identification modulo $I$. A choice of this $G$-torsor isomorphism over $B_\dR^+(A')$ allows us to transfer the descent datum of $\cE'$ to a descent datum of $\cE|_{B_\dR^+(A')}$. Exhibiting $\cE$ as a unique $G$-torsor descent of $\cE'$ becomes the task of finding a unique $G$-torsor automorphism $\gA$ of $\cE|_{B_\dR^+(A')}$ that reduces to the identity modulo $I$ such that precomposition with $\gA$ matches the transferred descent datum with the canonical descent datum of $\cE|_{B_\dR^+(A')}$. The $G$-torsor automorphisms of (base changes of) $\cE$ are parametrized by a reductive $B_\dR^+(A)$-group scheme $\cG$ that is a twist of $G$, in particular,
\[
\tst \cG(B_\dR^+(A')) \cong \varprojlim_{n > 0} \cG(B_\dR^+(A')/I^n) \qx{with surjective transition maps,}
\]
so that it suffices to uniquely construct $\gA$ inductively modulo $I^n$ for all $n > 0$. More precisely, using the surjectivity of the transition maps and the inductive hypothesis, it suffices to show that every descent datum of $\cE|_{B_\dR^+(A')/I^{n + 1}}$ that reduces to the canonical descent datum modulo $I^n$ may be transformed to the canonical descent datum of $\cE|_{B_\dR^+(A')/I^{n + 1}}$ by precomposing with a uniquely determined $G$-torsor automorphism of $\cE|_{B_\dR^+(A')/I^{n + 1}}$ that reduces to the identity modulo $I^n$. By deformation theory \cite{Ill05}*{Theorem~8.5.9~(a)}, the $G$-torsor automorphisms of $\cE|_{B_\dR^+(A')/I^{n + 1}}$ that reduce to the identity modulo $I^n$ are parametrized by $H^0(A', \Lie \cG \tensor_{A} I^n/I^{n +1})$, and similarly over $A''$, etc. Thus, the task becomes showing the following for the \v{C}ech cohomology groups of our~$v$-cover:
\[
\check{H}^1_v(A'/A, \Lie \cG \tensor_{A} I^n/I^{n +1}) \cong 0 \qxq{and}  \check{H}^0_v(A'/A, \Lie \cG \tensor_{A} I^n/I^{n +1}) \cong \Lie \cG \tensor_{A} I^n/I^{n +1}. 
\] 
Since $\Lie \cG \tensor_{A'} I^n/I^{n +1}$ is a finite projective $A$-module, both of these desired identifications follow  by combining the \v{C}ech-to-derived spectral sequence with the vanishing of the higher $v$-cohomology of the structure sheaf on affinoid perfectoids \cite{SW20}*{Theorem~17.1.3}.

In conclusion, $\cE$ is actually a unique descent of $\cE'$ relative to the descent datum that we started from. It remains to argue that $\iota'$ descends uniquely as well, for which it now suffices to show that 
\[
\cE(B_\dR(A)) \ra \cE(B_\dR(A')) \rightrightarrows \cE(B_\dR(A''))
\]
is an equalizer diagram. For this, since $\cE$ inherits affineness from $G$, it suffices to show that 
\[
B_\dR(A) \ra B_\dR(A') \rightrightarrows B_\dR(A'')
\]
is an equalizer diagram. Since the ideal $I$ is functorial and principal, it suffices to show the same for 
\[
B_\dR^+(A) \ra B_\dR^+(A') \rightrightarrows B_\dR^+(A'').
\]
This, however, may be checked modulo powers of $I$, where, since $I$ is generated by a nonzerodivisor, it follows from the structure presheaf being a $v$-sheaf on perfectoid spaces \cite{Sch22}*{Theorem~8.7}.
\end{proof}





\section{The analytic topology suffices} \label{sec:enough}

We are ready for our promised main result that the analytic sheafification suffices in the forming of the $B_\dR^+$-affine Grassmannian. This hinges on algebraization and approximation techniques from \cite{Hitchin-torsors}*{Section 2} and on the discrete valuation ring case of the Grothendieck--Serre~conjecture. 

\bthm\label{thm:main}
For a nonarchimedean local field $K$ and a reductive group scheme $G$ defined either over $K$ or over $\cO_K$ as in \uS\uref{sec:recall}, the $B_\dR^+$-affine Grassmannian $\Gr_G^{B_\dR^+}$ is the sheafification of the presheaf quotient $LG/L^+G$ with respect to the analytic topology on perfectoid  $\cO_K$-algebra pairs $(A, A^+)$. 
\ethm

We recall that the \emph{analytic topology} is the one whose covers $\{(A, A^+) \ra (A_j, A_j^+)\}_{j \in J}$ are characterized by the maps $\Spa(A_j, A_j^+) \ra \Spa(A, A^+)$ being jointly surjective open immersions.

\bpf
By \Cref{prop:Gr-is-v-sheaf}, the $B_\dR^+$-affine Grassmannian $\Gr_G^{B_\dR^+}$ is a sheaf for the analytic topology (even for the $v$-topology) and contains the presheaf quotient $LG/L^+G$ as a subfunctor that parametrizes those pairs $(\cE, \iota)$ in which $\cE$ is a trivial torsor. Thus, we only need to show that for every perfectoid  $\cO_K$-algebra pair $(A, A^+)$ and every $G$-torsor $\cE$ over $B_\dR^+(A)$ that becomes trivial over $B_\dR(A)$, each $x \in \Spa(A,A^+)$ lies in some rational open subset $\Spa(A',A^{\prime+}) \subset \Spa(A,A^+)$ such that $\cE|_{B_\dR^+(A')}$ is trivial. For this, since $B_\dR^+(A')$ is complete with respect to the kernel of its surjection onto $A'$, by the smoothness of $\cE$ inherited from $G$ (or by \cite[Theorem 2.1.6]{Hitchin-torsors}), it suffices to show that for each $x \in \Spa(A,A^+)$, the $G$-torsor $E \ce \cE|_A$ trivializes over some $A'$ as above. 

To find the desired $A'$, we first let $\Spa(A_i, A_i^+) \subset \Spa(A, A^+)$ range over all the rational open subsets containing $x$ and 
consider the perfectoid residue pair $(k_x, k_x^{+})$~with
\begin{equation*}
\tst    k_x^{+}\cong\left(\varinjlim_i A_i^{+}\right)^{\wedge} \qxq{and} k_x \cong k_x^+[\f1\varpi],
\end{equation*}
where $\varpi \in A^+$ is a pseudouniformizer and the completion is $\varpi$-adic. Each $A_i^+$ is $\varpi$-Henselian (even $\varpi$-adically complete) and $A_i \cong A_i^+[\f1\varpi]$, so $\varinjlim_i A_i^{+}$ is $\varpi$-Henselian (see \cite{SP}*{Lemma~\href{https://stacks.math.columbia.edu/tag/0FWT}{0FWT}}) with $\varinjlim_i A_i \cong (\varinjlim_i A_i^{+})[\f1\varpi]$. The algebraization and approximation \cite[Corollary 2.1.22~(c)]{Hitchin-torsors}\footnote{Or already its earlier version \cite{GR03}*{Theorem~5.8.14}.} 
gives
\[
\tst \varinjlim_i H^1(A_i, G) \cong  H^1(\varinjlim_i A_i, G) \isomto H^1(k_x,G)
\]
(for the first isomorphism in this display, see, for instance, \cite{poitou-tate}*{Lemma 2.1}).\footnote{In fact, we will only use the injectivity of the displayed map, which also follows from the Elkik-style approximation result \cite{Hitchin-torsors}*{Theorem 2.2.2}, or already from its earlier version \cite{GR03}*{Proposition 5.4.21}.} Thanks to this, all that remains is to show that the $G$-torsor $E|_{k_x}$ is trivial granted that we know that it lifts to a $G$-torsor over $B_\dR^+(k_x)$, namely, to $\cE|_{B_\dR^+(k_x)}$, that becomes trivial over $B_\dR(k_x)$. 

However, $B_\dR^+(k_x)$ is a discrete valuation ring and $G$ is reductive, so the Nisnevich case of the Grothendieck--Serre conjecture \cite{Nis82}*{Chapter II, Theorem 4.2}, \cite{Nis84}*{th\'{e}or\`eme 2.1} (see also \cite{Guo22a}*{Theorem 1}) implies that no nontrivial $G$-torsor over $B_\dR^+(k_x)$ trivializes over $B_\dR(k_x)$.
\epf



\brem
The only way in which we used the assumption that our group $G$ is reductive, as opposed to, say, merely smooth and affine, is to apply the Grothendieck--Serre type result that no nontrivial $G$-torsor over the discrete valuation ring $B_\dR^+(k_x)$ trivializes over $B_\dR(k_x)$. Granted that this input is obtained, the same argument would give \Cref{thm:main} for other classes of smooth affine groups, for instance, it would be interesting to know whether the same holds for parahoric $G$.
\erem

\subsection*{Acknowledgements}  
This project started during the first-named author's visit to the University of Tokyo. He thanks the University of Tokyo, especially, Naoki Imai, for hospitality during this visit.
 We thank the referee, as well as Bhargav Bhatt and Arpon Raksit, for helpful comments. This project has received funding from the European Research Council (ERC) under the European Union's Horizon 2020 research and innovation programme (grant agreement No.~851146). This project is based upon work supported by the National Science Foundation under Grant No.~DMS-1926686. During the writing of this paper the second named author was a JSPS fellow, supported by a KAKENHI Grant-in-aid (grant number 22F22323).



\end{document}